\documentclass[10pt,electronic]{amsart}        
\usepackage{tikz-cd}
\usepackage[english]{babel}
\usepackage[hmargin=3.7cm,vmargin=3.5cm]{geometry} 
\usepackage{amsmath}
\usepackage{amsthm}                      
\usepackage{amssymb}
\usepackage{mathrsfs}
\usepackage{enumerate}

\usepackage{euscript}
\renewcommand{\mathcal}{\EuScript}

\theoremstyle{plain}                                                           
\newtheorem{thm}{Theorem}[section]
\newtheorem{lem}[thm]{Lemma}
\newtheorem{prop}[thm]{Proposition}

\newtheorem{hyp}[thm]{Hypothesis}
\theoremstyle{definition}
\newtheorem{examplex}[thm]{Example}
\newenvironment{ex}
{\pushQED{\qed}\renewcommand{\qedsymbol}{\lower-0.3ex\hbox{$\triangleleft$}}\examplex}
{\popQED\endexamplex}
\newtheorem{rem}[thm]{Remark}

\newcommand{\F}{\mathcal F}
\DeclareMathOperator{\rank}{rank}
\DeclareMathOperator*{\colim}{colim}
\newcommand{\sgn}{\mathrm{sgn}}
\newcommand{\N}{\mathrm N}
\newcommand{\gr}{\mathrm{Gr}}
\newcommand{\operad}[1]{\mathsf{#1}}
\newcommand{\Com}{\operad{Com}}
\newcommand{\field}[1]{\ensuremath{\mathbf{#1}}}
\newcommand{\Q}{\ensuremath{\field{Q}}}        
\newcommand{\C}{\ensuremath{\field{C}}}
\newcommand{\sym}{\ensuremath{\mathbb{S}}}
\newcommand{\Z}{\ensuremath{\field{Z}}} 
\newcommand{\A}{\mathcal{A}}
\newcommand{\MM}{\overline{\mathcal{M}}}
\newcommand{\R}{\mathbf{R}}
\renewcommand{\H}{\mathsf H}
\newcommand{\MHM}{\mathsf{MHM}}

\renewcommand{\S}{\mathsf S}
\newcommand{\D}{\mathbb D}
\newcommand{\B}{\mathbf{B}}

\newcommand{\FI}{\mathrm{FI}}

\title[A spectral sequence for stratified spaces]{A spectral sequence for stratified spaces and configuration spaces of points}

\author{Dan Petersen}
\thanks{}
\email{danpete@math.ku.dk}
\address{}

\begin{document} 
 \maketitle   
 
 \begin{abstract}We construct a spectral sequence associated to a stratified space, which computes the compactly supported cohomology groups of an open stratum in terms of the compactly supported cohomology groups of closed strata and the reduced cohomology groups of the poset of strata. Several familiar spectral sequences arise as special cases. The construction is sheaf-theoretic and works both for topological spaces and for the \'etale cohomology of algebraic varieties. As an application we prove a very general representation stability theorem for configuration spaces of points.
 \end{abstract}

\section{Introduction}

Let $X = \bigcup_{\alpha \in P} S_\alpha$ be a stratified space. By this we mean that the topological space $X$ is the union of disjoint locally closed subspaces $S_\alpha$ called the \emph{strata}, and that the closure of each stratum is itself a union of strata. The set $P$ of strata becomes partially ordered by declaring that $\alpha \leq \beta$ if $\overline S_\alpha \supseteq S_\beta$. 

Let $\chi_c(-)$ denote the compactly supported Euler characteristic of a space. Since this invariant is additive over stratifications, one has an equality
\begin{equation}\label{first}
\chi_c(\overline S_\alpha) = \sum_{\alpha \leq \beta} \chi_c(S_\beta)
\end{equation}
for all $\alpha \in P$. By the M\"obius inversion formula for the poset $P$, it therefore holds that
\begin{equation}\label{second}
\chi_c(S_\alpha) = \sum_{\alpha \leq \beta} \mu_P(\alpha,\beta) \cdot \chi_c(\overline S_\beta)
\end{equation}
where $\mu_P$ is the M\"obius function of the poset. This expresses the simple combinatorial fact that if one knows all the integers $\chi_c(\overline S_\alpha)$, then one can also determine the integers $\chi_c(S_\alpha)$ by inclusion-exclusion. 

Equation \eqref{first} can be upgraded (or ``categorified'') to a relationship between actual cohomology groups. Suppose $\sigma \colon P \to \Z$ is a function such that $\sigma(\alpha) < \sigma(\beta)$ if $\alpha < \beta$. Such a function defines a filtration of $\overline S_\alpha$ by closed subspaces, and the corresponding spectral sequence in compactly supported cohomology reads
\begin{equation} \label{filtrationsseq}E_1^{pq}=\bigoplus_{\substack{\alpha \leq \beta \\ \sigma(\beta)=-p}}H^{p+q}_c(S_\beta,\Z) \implies H^{p+q}_c(\overline S_\alpha,\Z) \end{equation}
By equating the Euler characteristics of the $E_1$ and $E_\infty$ pages of this spectral sequence one recovers Equation \eqref{first}. 

It is then natural to ask whether also the dual Equation \eqref{second} admits a similar interpretation. We point out that the quantity $\mu_P(\alpha,\beta)$ is \emph{also} an Euler characteristic, by Philip Hall's theorem: the M\"obius function $\mu_P(\alpha,\beta)$ equals the reduced Euler characteristic of $\N(\alpha,\beta)$, by which we mean the nerve of the poset $(\alpha,\beta)$, where $(\alpha,\beta)$ denotes an open interval in $P$. (The preceding is valid only if $\alpha < \beta$: in the degenerate case $\alpha=\beta$ it is natural to define the reduced cohomology of $\N(\alpha,\beta)$ to be $\Z$ in degree $-2$, as we explain in Section \ref{posetsection}.) In any case, one can expect such a categorification to also involve the reduced cohomology groups of the poset. 

In this paper, we construct a spectral sequence accomplishing this goal:

\begin{thm}\label{thmA}
	There exists a spectral sequence 
	$$E_1^{pq}=\bigoplus_{\substack{\alpha \leq \beta \\ \sigma(\beta)=p}} \bigoplus_{i+j+2=p+q}H^{j}_c(\overline S_\beta,\widetilde H^i(\N (\alpha,\beta),\Z)) \implies H^{p+q}_c(S_\alpha,\Z). $$
\end{thm}

Taking Euler characteristics of both sides, we recover Equation \eqref{second}. This would seem a very natural question --- given the cohomology of the closed strata, how does one compute the cohomology of open strata? --- and it is close in spirit to the work of Vassiliev \cite{vassilievbook,vassilievicm}. Yet to my knowledge the result is new.

The proof is elementary and completely sheaf-theoretic, and the theorem we prove in the body of the paper is a more general statement that is valid with coefficients given by any sheaf or complex of sheaves $\F$ on $X$. It also works in the setting of $\ell$-adic sheaves, if $X$ is an algebraic variety: in this case, the spectral sequence is a spectral sequence of $\ell$-adic Galois representations. 

As an application of our result we prove a very general representation stability theorem for configuration spaces of points. In particular, a novel feature is that if one is willing to work with Borel--Moore homology (or dually, compact support cohomology), then one can prove homological stability results for an arbitrary topological space $M$ satisfying rather mild hypotheses; to my knowledge, all existing results in the literature prove homological stability for configuration spaces of points on \emph{manifolds}. In this introduction we focus on the case when $M$ is a (possibly singular) algebraic variety, in which case the result is easier to state. 

Let $M$ be a space, and let $\A$ be a finite collection of closed subspaces $A_i \subset M^{n_i}$. We define a configuration space $F_\A(M,n)$, parametrizing $n$ ordered points on $M$ ''avoiding all $\A$-configurations''. For instance, if $\A$ consists only of the diagonal inside $M^2$, then $F_\A(M,n)$ is the usual configuration space of distinct ordered points on $M$. %Let $H_\bullet^{BM}(-)$ denote the Borel--Moore homology groups of a space.

\begin{thm}\label{thmB}Let $M$ be a geometrically irreducible $d$-dimensional algebraic variety over a field $\kappa$, and let $\A$ be an arbitrary finite collection of closed subvarieties $A_i \subset M^{n_i}$. 
	\begin{enumerate}
		\item For $\kappa=\C$, the (singular) Borel--Moore homology groups $H_{i+2dn}^{BM}(F_\A(M,n)(\C),\Z)$ form a finitely generated FI-module for all $i \in \Z$. 
		\item The (\'etale) Borel--Moore homology groups $H_{i+2dn}^{BM}(F_\A(M,n)_{\overline \kappa},\Z_\ell(-dn))$ form a finitely generated FI-module in $\ell$-adic $\mathrm{Gal}(\overline\kappa/\kappa)$-representations, for all $i \in \Z$, whenever $\ell$ is a prime different from $\mathrm{char}(\kappa)$.
	\end{enumerate}
\end{thm}

In particular, the homology groups $H_{i+2dn}^{BM}(F_\A(M,n),\Q)$ form a  representation stable sequence of $\sym_n$-representations, and the $\sym_n$-invariants $H_{i+2dn}^{BM}(F_\A(M,n)/\sym_n,\Q)$ satisfy homological stability as $n \to \infty$. 

If $M$ is smooth, or at least a homology manifold, we may conclude instead that the cohomology groups $H^{i}(F_\A(M,n),\Z)$ form a finitely generated FI-module, by Poincar\'e duality. See Remark \ref{homologymanifoldremark}. 

The fact that we obtain a finitely generated FI-module over $\Z$ gives a homological stability result with rational coefficients, but it also has interesting consequences for the mod $p$ homology of the unordered configuration spaces: by results of Nagpal \cite{nagpalmodular}, our theorem implies that the groups $H_{i+2dn}^{BM}(F_\A(M,n)/\sym_n,\mathbf F_p)$ become eventually periodic. 

Vakil and Wood \cite{vakilwood} introduced certain configuration spaces $\overline w_\lambda^c(M)$, depending on a partition $\lambda$. For a suitable choice of $\A$, one has $F_\A(M,n)/\sym_n = \overline w_\lambda^c(M)$, so Theorem \ref{thmB} implies in particular a homological stability theorem for the spaces $\overline w_\lambda^c(M)$ as $n \to \infty$, which proves a conjecture of Vakil and Wood \cite[Conjecture F]{vakilwood}. This conjecture has previously been proven by Kupers--Miller--Tran \cite{kupersmillertran}. Compared to Kupers--Miller--Tran, our proof gives the stronger assertion of representation stability, and makes no smoothness assumptions about $M$ (they assume $M$ is a smooth manifold). On the other hand, their proof gives in many cases integral stability for the unordered configuration space, and they give an explicit stability range. The latter should be possible in our setting, too, but we have not done so.

A remark is that Vakil and Wood formulated their conjecture after making point counts of varieties over finite fields, and using the Grothendieck--Lefschetz trace formula to guess what the cohomology should look like. Since the Grothendieck--Lefschetz trace formula concerns \emph{compact support} cohomology, it is in a sense natural that we obtain stronger results when working with compact support cohomology/Borel--Moore homology from the start.

As mentioned, one can prove also a version of Theorem \ref{thmB} for an arbitrary topological space, but the assumptions on $M$ and $\A$ become more cumbersome to state. However, if we let $\A$ be the arrangement leading to the configuration spaces considered by Vakil--Wood, the hypotheses are quite simple: if $M$ is any locally compact topological space with finitely generated Borel--Moore homology groups, and such that there exists an integer $d \geq 2$ for which $H_d^{BM}(M,\Z) \cong \Z$ and $H_i^{BM}(M,\Z)=0$ for $i>d$, then $H_{i+dn}^{BM}(F_\A(M,n),\Z)$ is a finitely generated FI-module if $d$ is even; if $d$ is odd, one needs to twist by the sign representation, and $H_{i+dn}^{BM}(F_\A(M,n),\Z) \otimes \mathrm{sgn}_n$ is a finitely generated FI-module. 

\section{Generalities on posets}
\label{posetsection}
%\subsection{Nerve and homology of a poset}

Let $P$ be a poset, always assumed to be finite. We define its \emph{nerve}  $\N P$ to be the simplicial complex with vertices the elements of $P$, and a subset $S \subseteq P$ forms a face if and only if all elements of $S$ are pairwise comparable. The corresponding simplicial set is exactly the usual nerve of $P$, when $P$ is thought of as a category.

We use $\widetilde C_\bullet(\Delta)$ to denote the \emph{reduced cellular chains} of a simplicial complex $\Delta$. The group $\widetilde C_i(\Delta)$ is free abelian on the set of $i$-dimensional faces; we include the empty set as a $(-1)$-dimensional face. The homology of this chain complex  is $\widetilde H_\bullet(\Delta,\Z)$. However, we will prefer to work with cohomology. The usual definition would be to set $$\widetilde C^\bullet(\Delta) = \mathrm{Hom}(\widetilde C_\bullet(\Delta),\Z)$$
but we will find it more convenient to use the distinguished basis of $\widetilde C_i(\Delta)$ to consider $\widetilde C^i(\Delta)$ as \emph{also} being the free abelian on the set of $i$-dimensional faces; then the differential becomes an alternating sum over ways of \emph{adding} an element to a face.

If $x\leq y$ in $P$, we denote by $\widetilde C^\bullet(x,y)$ the chain complex which in degree $d$ is the free abelian group spanned by the increasing sequences
$$ x = z_{-1} < z_0 < z_1 < \ldots < z_{d} < z_{d+1} = y,$$
and whose differential $\partial : \widetilde C_d \to \widetilde C_{d+1}$ is an alternating sum over ways of adding an element to the sequence. For $x<y$, $\widetilde C^\bullet(x,y)$ is equal to $\widetilde C^\bullet(\N(x,y))$, where $(x,y)$ denotes an open interval in $P$; for $x=y$, it consists of $\Z$ placed in degree $-2$. We denote by $\widetilde H^\bullet(x,y)$ the cohomology of this cochain complex.

%\subsection{Construction of a spectral sequence}
Let $\sigma \colon P \to \Z$ be a strictly increasing function, e.g. a grading or a linear extension. 

\begin{prop}\label{posetseq}For $x < y$ in $P$, there exists a spectral sequence
$$ E_1^{pq} = \bigoplus_{\substack{x \leq z \leq y\\ \sigma(z)=p}} \widetilde H^{p+q-1}(x,z)$$
converging to zero.\end{prop}

\begin{proof}
Consider the chain complex $\widetilde C^\bullet=\widetilde C^\bullet (\N(x,y])$, where $(x,y]$ denotes a half-open interval in $P$. Since $(x,y]$ has a unique maximal element its nerve is contractible, so the complex $\widetilde C^\bullet$ is acyclic. We identify $\widetilde C^d$ with the set of increasing sequences
$$ x = z_{-1} < z_0 < z_1 < \ldots < z_{d} \leq y.$$

Define a decreasing filtration on this complex by taking $F^p \widetilde C^d$ to be the span of all sequences with $\sigma(z_d) \geq p$. This makes $\widetilde C^\bullet$ a filtered complex. Consider the quotient
$$ F^p \widetilde C^\bullet/F^{p+1}\widetilde C^\bullet$$
and its induced differential. Then $F^p \widetilde C_d/F^{p+1}\widetilde C_d$ has a basis consisting of sequences such that $\sigma(z_d)$ is \emph{exactly} equal to $p$, and the differential is a sum over all ways of adding an element to the sequence coming \emph{before} $z_d$. It therefore follows that the quotient is isomorphic to the direct sum 
	$$ \bigoplus_{\substack{x \leq z \leq y\\ \sigma(z)=p}} \widetilde C^{\bullet-1}(x,z),$$
by an isomorphism taking the sequence $x = z_{-1} < z_0 < z_1 < \ldots < z_{d} \leq y \in F^p\widetilde C^{d}/F^{p+1}\widetilde C^d$ to the sequence $x = z_{-1} < z_0 < z_1 < \ldots z_{d-1} < z_{d} = z_d$ in $\widetilde C^{d-1}(x,z_d)$.
Thus the spectral sequence associated to this filtration has the required form.	
\end{proof}

%\subsection{The M\"obius function}
\newcommand{\Int}{\mathrm{Int}} 
Let $\Int(P)$ denote the set of pairs $(x,y) \in P \times P$ with $x \leq y$. We define the \emph{M\"obius function}
$$ \mu \colon \Int(P) \to \Z$$
by $\mu(x,y) = \sum_i (-1)^i \rank \widetilde H^i(x,y)$.
\begin{prop}
	If $x <y$, then $\sum_{x \leq z \leq y} \mu(x,z) = 0$.
\end{prop}

\begin{proof}
	The left hand side is the Euler characteristic of the $E_1$ page of the spectral sequence constructed in Proposition \ref{posetseq}, and the right hand side is the Euler characteristic of the $E_\infty$ page. 
\end{proof}

A consequence of the above is a simple recursive procedure for calculating the M\"obius function: the M\"obius function could equivalently have been defined as 
$$\mu(x,y) = 
\begin{cases} 
1 & \text{if } x = y \\
-\sum_{x \leq z < y} \mu(x,z) & \text{if } x < y 
\end{cases}$$
In most treatments this is taken as the definition of the M\"obius function. The fact that $\mu(x,y)$ for $x<y$ equals the reduced Euler characteristic of the nerve of the interval $(x,y)$ is then called Philip Hall's theorem. We may think of Proposition \ref{posetseq} as a categorification of the usual recursion for the M\"obius function.

\section{The construction and several examples}
\label{construction}
Let $X = \bigcup_{\alpha \in P} S_\alpha$ be a stratified space. By this we mean that the space $X$ is the union of disjoint locally closed subspaces $S_\alpha$ called the \emph{strata}, and that the closure of each $S_\alpha$ is itself a union of strata. By a ``space'' we mean \emph{either}:
\begin{enumerate}
	\item $X$ is a locally compact Hausdorff topological space,
	\item $X$ is an algebraic variety over some field.
\end{enumerate}
In the former case, ``sheaf'' will just mean ``sheaf of abelian groups''; in the latter case, ``sheaf'' will mean ``constructible $\ell$-adic sheaf, for $\ell$ different from the characteristic''.

The set $P$ of strata becomes partially ordered by declaring that $\alpha \leq \beta$ if $\overline S_\alpha \supseteq S_\beta$. We assume for simplicity (and without loss of generality) that $P$ has a unique minimal element $0$, i.e.\ a unique open dense stratum $S_0$. 

For $\alpha \in P$ we denote by $j_\alpha$ the locally closed inclusion $S_\alpha \hookrightarrow X$, and by $i_\alpha$ the inclusion $\overline S_\alpha \hookrightarrow X$ of the closure of a stratum. 

%\subsection{Constructing the resolution}

For $d \geq 0$, define a sheaf 
$$ L^d(\F) = \bigoplus_{0 = \alpha_0 < \alpha_1 < \ldots < \alpha_d \in P}  (i_{\alpha_d})_\ast (i_{\alpha_d})^\ast  \F$$
on $X$. In particular, $L_0(\F) = \F$. We may define a differential
$$ L^d(\F) \rightarrow L^{d+1}(\F)$$
as an alternating sum over ways of adding an element to the sequence $\alpha_0 < \alpha_1 < \ldots < \alpha_d$, just as in our definition of $\widetilde C^\bullet(\Delta)$ for a simplicial complex $\Delta$; when the element we add appears at the end of the sequence, i.e., as $\alpha_{d+1} > \alpha_d$, then the differential uses the map 
$$ (i_{\alpha_d})_\ast (i_{\alpha_d})^\ast  \F \to (i_{\alpha_{d+1}})_\ast (i_{\alpha_{d+1}})^\ast  \F$$
obtained from the fact that $i_{\alpha_{d+1}}$ factors through $i_{\alpha_{d}}$. This makes $L^\bullet(\F)$ into a complex, for the same reason that $\widetilde C^\bullet(\Delta)$ is. 

\begin{prop}
The complex $L^\bullet(\F)$ is quasi-isomorphic to $j_{0!}^{}j^{-1}_0\F$, where $j_0$ is the inclusion of the open stratum.  
\end{prop}

\begin{proof}
	We show that
	$$ j_{0!}^{}j^{-1}_0\F \to L^0(\F) \to L^1(\F) \to L^2(\F) \to \ldots $$
	is an acyclic complex of sheaves. It suffices to check this on stalks. For $x \in S_0$ the induced sequence on stalks reads
	$$ \F_x \to \F_x \to 0 \to 0 \to \ldots, $$
	with the map $\F_x \to \F_x$ the identity. Thus we may restrict attention to $x$ in some stratum $S_\beta$, $\beta \neq 0$. In this case $(j_{0!}^{}j^{-1}_0\F)_x$ will of course vanish, and we will have
	$$L^d(\F)_x \cong \bigoplus_{\substack{0 = \alpha_0 < \alpha_1 < \ldots < \alpha_d \in P\\\alpha_d \leq \beta}} \F_x $$
	with the differential on $L^\bullet(\F)_x$ given by adding an element to the sequence of $\alpha_i$'s. But this means that $L^\bullet(\F)_x$ is (up to a degree shift) the tensor product of $\F_x$ with the complex $\widetilde C^\bullet(\N (0,\beta])$, which is acyclic because the poset $(0,\beta]$ has a unique maximal element.
	\end{proof}

\begin{rem}
	Another way to think about the complex $L^\bullet(\F)$ is that $j_{0!}^{}j^{-1}_0\F[-1]$ can be calculated as the cone of $\F \to i_\ast i^\ast \F$, where $i \colon (X \setminus S_0) \to X$ is the inclusion. Now $i_\ast i^\ast \F$ may be calculated as the homotopy limit of the various $(i_{\alpha})_\ast (i_{\alpha})^\ast  \F$ (for $\alpha \neq 0$), and $L^{\geq 1}(\F)$ is the bar resolution computing this homotopy limit. %This goes back to \cite{roos,nobeling}.
\end{rem}
%
%\begin{rem}
%	Our definition of $L^\bullet(\F)$ is inspired by the standard complex calculating the higher derived functors of the inverse limit 
%\end{rem}

%\subsubsection{The spectral sequence}

Suppose we are given $\sigma \colon P \to \Z$ an increasing function. We may now define a decreasing filtration of $L^\bullet(\F)$ by taking 
$$ F^p L^d(\F) = \bigoplus_{\substack{0 = \alpha_0 < \alpha_1 < \ldots < \alpha_d \in P \\ \sigma({\alpha_d}) \geq p}}  (i_{\alpha_d})_\ast (i_{\alpha_d})^\ast  \F.$$
The compactly supported hypercohomology spectral sequence associated to this filtration reads
$$ E_1^{pq} = \mathbb H^{p+q}_c(X,\mathrm{Gr}_F^{p} L^\bullet(\F)) \implies \mathbb H^{p+q}_c(X,L^\bullet(\F)) = H^{p+q}_c(S_0,j_0^{-1}\F)$$
where the second equality is the preceding proposition. Thus we should understand the associated graded $\mathrm{Gr}_F^{p} L^\bullet(\F)$. By  arguments just like those in the proof of Proposition \ref{posetseq}, the associated graded can be written as
$$\mathrm{Gr}_F^{p} L^d(\F) = \bigoplus_{\substack{\beta \in P\\ \sigma(\beta)=p}} \bigoplus_{{0 = \alpha_0 < \alpha_1 < \ldots < \alpha_d = \beta}} (i_\beta)_\ast (i_\beta)^\ast \F, $$
and hence
$$ \mathrm{Gr}_F^{p} L^\bullet(\F) = \bigoplus_{\substack{\beta \in P\\ \sigma(\beta)={p}}} \widetilde C^{\bullet+2}(0,\beta) \otimes (i_\beta)_\ast (i_\beta)^\ast \F = \bigoplus_{\substack{\beta \in P\\ \sigma(\beta)={p}}} \widetilde H^{\bullet+2}(0,\beta) \otimes^L (i_\beta)_\ast (i_\beta)^\ast \F.$$
The last equality uses that $\widetilde C^{\bullet+2}(0,\beta)$ is a complex of free modules, so it calculates the derived tensor product, and that any complex of abelian groups is quasi-isomorphic to its cohomology. 

In full generality, this can not be simplified further. However, in most cases occuring in practice we can:
\begin{enumerate}
	\item If $\F = R$ is a constant sheaf associated to the ring $R$, then $\widetilde H^{\bullet+2}(0,\beta) \otimes^L (i_\beta)_\ast (i_\beta)^\ast \F$ is the constant sheaf $\widetilde H^{\bullet+2}(0,\beta;R)$. 
	\item If the cohomology groups $\widetilde H^{i}(0,\beta)$ are torsion free, or if $\F$ is a sheaf of $k$-vector spaces for some field $k$, then we may replace the derived tensor product with the usual tensor product. 
\end{enumerate}
Let us state our main result only in these two simpler situations.

\begin{thm}\label{mainthm}Let $X = \bigcup_{\beta \in P} S_\beta$ be a stratified space, where the set $P$ is partially ordered by reverse inclusion of the closures of strata. Choose a function $\sigma : P \to \Z$ such that $\sigma(x)<\sigma(y)$ if $x<y$. 
	
	\begin{enumerate}[(i)]
		\item 
		For any ring $R$, there is a spectral sequence 
		$$ E_1^{pq} = \bigoplus_{\substack{\beta \in P\\ \sigma(\beta)={p}}} \bigoplus_{i+j=p+q-2}  H^j_c(\overline S_\beta,\widetilde H^{i}(0,\beta;R)) \implies H^{p+q}_c(S_0,R). $$ 
		\item If $\F$ is a sheaf on $X$, and we assume either that $\F$ is a sheaf of $k$-vector spaces or that each interval $(0,\beta)$ in $P$ has torsion free cohomology, then the there is a spectral sequence 
		$$ E_1^{pq} = \bigoplus_{\substack{\beta \in P\\ \sigma(\beta)={p}}} \bigoplus_{i+j=p+q-2}  \widetilde H^{i}(0,\beta) \otimes H^j_c(\overline S_\beta, i_\beta^\ast \F) \implies H^{p+q}_c(S_0,j_0^{-1}\F). $$
		\end{enumerate}
	If in (ii) $X$ is an algebraic variety and $\F$ is an $\ell$-adic sheaf, then this spectral sequence is a spectral sequence of Galois representations, if the cohomology groups $\widetilde H^{i}(0,\beta)$ are given the trivial Galois action.
\end{thm}

\subsection{Examples and applications}

\begin{ex}
	If the stratification consists only of a closed subspace $i : Z \hookrightarrow X$, then the complex $L^\bullet(\F)$ reduces to the two-term complex $\F \to i_\ast i^\ast \F$, and the spectral sequence reduces to the long exact sequence
	\[ \ldots \to H^k_c(X,\F) \to H^k_c(Z,\F) \to H^{k+1}_c(X \setminus Z,\F) \to H^{k+1}_c(X,\F) \to \ldots \qedhere \]
\end{ex}

\begin{ex}
	Let $X$ be a complex manifold, and $D= D_1 \cup \ldots \cup D_k$ a strict normal crossing divisor. Consider the stratification of $X$ by the various intersections of the components of $D$. For $I \subset \{1,\ldots,k\}$, let $D_I = \bigcap_{i\in I} D_i$, including $D_\emptyset = X$.  Each interval in the poset of strata is a boolean lattice, so its reduced cohomology vanishes below the top degree, where it is one-dimensional. The spectral sequence therefore reduces to 
	$$ E_1^{pq} = \bigoplus_{\vert I \vert = p} H^q_c(D_I,\Z) \implies H^{p+q}_c(X \setminus D,\Z). $$
	This is the Poincar\'e dual of the spectral sequence used by Deligne to construct the mixed Hodge structure on a smooth noncompact complex algebraic variety \cite{hodge2}. In the algebraic case, the above is a spectral sequence of mixed Hodge structures/Galois representations. 
\end{ex}

\begin{ex}
	Suppose $X = \AA^n$ is affine space over a field, and the stratification consists of all the intersections in some subspace arrangement. Let $\sigma(\alpha) = -\dim(S_\alpha)$. Let $\F=\Q_\ell$. In this case the spectral sequence simplifies to 
	$$ E_1^{pq} = \bigoplus_{\substack{\beta \in P\\ \sigma(\beta)={p}}} \widetilde H^{q+3p-2}(0,\beta) \otimes \Q_\ell(p) \implies H^{p+q}_c(S_0,\Q_\ell).$$ 
	Since all columns have different weight there can be no differentials in the spectral sequence. It follows that (up to semisimplification of the Galois representation)
	$$ H^n_c(S_0,\Q_\ell) = \bigoplus_{j=0}^n \Q_\ell(-j) \otimes \left( \bigoplus_{\beta : \dim(S_\beta)=j} \widetilde H^{n-2j-2}(0,\beta) \right)$$
	which re-proves a result of Bj\"orner and Ekedahl \cite{bjornerekedahlsubspace}. 
\end{ex}

\begin{ex}The aforementioned result of Bj\"orner and Ekedahl is the algebro-geometric analogue of a theorem of Goresky--MacPherson \cite{goreskymacphersonstratifiedmorsetheory} about real subspace arrangements; the latter result, too, can be given an easy proof using our spectral sequence. Goresky and MacPherson originally proved it as an application of their stratified Morse theory; many different authors have subsequently given alternative proofs and/or strengthenings. Their result, in turn, is a refinement of the work of Orlik--Solomon on complex hyperplane arrangements \cite{orliksolomon}. In any case, suppose that $X$ is a real vector space, stratified according to intersections in a real subspace arrangement. Let $\sigma(\alpha) = -\dim_\R S_\alpha$. The result of Goresky--MacPherson is equivalent to our spectral sequence degenerating at $E_1$. The weight argument used in the case of a complex subspace arrangement is of course not valid in this setting. We can instead argue as follows: 
	
	Choose for each $\alpha$ an open ball $U_\alpha$ inside $S_\alpha$. Then $C_c^\bullet(U_\alpha,\Z)$ (compactly supported cochains) is a subcomplex of $C_c^\bullet(\overline S_\alpha,\Z)$ for all $\alpha$. The inclusion of each of these subcomplexes is a quasi-isomorphism, and the restriction maps between these subcomplexes are identically zero (since $U_\alpha$ is disjoint from all $S_\beta$ with $\beta > \alpha$.) By additionally choosing an arbitrary quasi-isomorphism between $C_c^\bullet(U_\alpha,\Z)$ and $H_c^\bullet(\overline S_\alpha,\Z)$ we thus get a quasi-isomorphism between the two functors $P \to \mathrm{Ch}_k$ (where the poset $P$ is thought of as a category) given by $\alpha \mapsto S^\bullet_c(\overline S_\alpha,\Z)$ and $\alpha \mapsto H^\bullet_c(\overline S_\alpha,\Z)$. 
	
	We can compute $R\Gamma_c(X,L^\bullet(\Z))$ by means of a double complex, with each vertical row a direct sum of complexes $S_c^\bullet(\overline S_\alpha,\Z)$, and the differentials in the horizontal row given by the differentials in the complex $L^\bullet(\Z)$; equivalently, given by the functor $\alpha \mapsto S^\bullet_c(\overline S_\alpha,\Z)$. If we apply the quasi-isomorphism of functors constructed in the previous paragraph we can replace this double complex with one in which the vertical rows have zero differential, and the horizontal rows are direct sums of complexes of the form 
	$\widetilde C^{\bullet+2}(0,\beta)$.
	
	Our spectral sequence arises from a filtration of this double complex. In this case, the filtration clearly splits, and the spectral sequence degenerates immediately. 
\end{ex}

\begin{ex}\label{variants}
	Let us give two variations of our spectral sequence.
	
	(${1^\circ}$) Let $\D$ denote Verdier's duality functor. The filtration on $L^\bullet(\F)$ induces a filtration on $\D L^\bullet(\F)$, satisfying $\mathrm{Gr}_F \D L^\bullet(\F) \simeq \D \mathrm{Gr}_F L^\bullet(\F)$ (see e.g.\ \cite[(2.2.8.1)]{illusiecotangent}). Thus the associated graded pieces of $\D L^\bullet(\F)$ are quasi-isomorphic to 
	$$\D\left(\bigoplus_{\sigma(\beta)=p}\widetilde H^{\bullet+2}(0,\beta) \otimes^L (i_\beta)_\ast (i_\beta)^\ast\F\right) = \bigoplus_{\sigma(\beta)=p}\widetilde H_{-\bullet-2}(0,\beta) \otimes^L (i_\beta)_! (i_\beta)^! \D \F.$$ Since the cohomology of $\D R$ in negative degrees equals Borel--Moore homology with coefficients in $R$, our filtration of $\D L^\bullet(\F)$ gives rise to a Borel--Moore homology spectral sequence
	$$E^1_{pq} = \bigoplus_{\substack{\beta \in P\\ \sigma(\beta)={p}}} \bigoplus_{i+j=p+q-2}  H_j^{BM}(\overline S_\beta,\widetilde H_{i}(0,\beta;R)) \implies H_{p+q}^{BM}(S_0,R).$$
	($2^\circ$) Instead of taking the compact support cohomology of $L^\bullet(\F) \simeq j_{0!}^{}j_0^{-1}\F$, we may take the usual cohomology. Since $j_{0!}^{}j_0^{-1}\F[1]$ is the cone of $\F \to i_\ast i^\ast \F$, where $i$ is the inclusion $(X \setminus S_0) \hookrightarrow X$, this gives instead a spectral sequence
	\[ E_1^{pq} = \bigoplus_{\substack{\beta \in P\\ \sigma(\beta)={p}}} \bigoplus_{i+j=p+q-2}  H^j(\overline S_\beta,\widetilde H^{i}(0,\beta;R)) \implies H^{p+q}(X,X\setminus S_0;R). \qedhere\]

\end{ex}

\begin{ex}
	Let $X = \{X(n)\}$ be a topological operad. Suppose that $X(n)$ is stratified in such a way that the strata correspond to trees with $n$ legs, the closed stratum correspond to a tree $T$ is $\prod_{v \in \mathrm{Vert}(T)} X(n_v)$, and the composition maps in the operad $X$ are given tautologically by grafting of trees. Let $Y(n)$ be the open stratum in $X(n)$ corresponding to the unique tree with a single vertex. 	Examples of such operads abound: the Stasheff associahedra, the Fulton--MacPherson model of the $\mathbf e_n$-operads, the Deligne--Mumford spaces $\MM_{0,n}$, the Boardman--Vogt $W$-construction applied to an arbitrary topological operad, Devadoss's mosaic operad, the cactus operad, Brown's dihedral moduli spaces $\mathcal M_{0,n}^\delta$, the brick operad $\mathcal B(n)$ of Dotsenko--Shadrin--Vallette...
	
	Clearly, the compact support cohomology $H^\bullet_c(X(n),\Q)$ will form a cooperad. Moreover, the degree-shifted cohomologies $H^{\bullet-1}_c(Y(n),\Q)$ will form an operad: $Y(n) \times Y(m)$ will be a stratum adjacent to $Y(n+m-1)$ inside $X(n+m-1)$, and there is a connecting homomorphism $H^\bullet_c(Y(n) \times Y(m),\Q) \to H^{\bullet+1}_c(Y(n+m-1),\Q)$ coming from the long exact sequence of a pair in compact support cohomology.
	
	If $\sigma$ is the function taking a stratum to the number of vertices in the corresponding tree, then we get a filtration of $X(n)$. The corresponding spectral sequence in compact support cohomology (Eq. \eqref{filtrationsseq} from the introduction) is a cooperad in the category of spectral sequences. Its $E_\infty$ page is the associated graded for a filtration on $H^\bullet_c(X(n),\Q)$, and its $E_1$ page is exactly the \emph{bar construction} on the operad $H^{\bullet-1}_c(Y(n),\Q)$. This construction seems to have first been considered in \cite[Subsection 3.3]{getzlerjones}, where it was used to prove Koszul self-duality of the $\mathbf e_n$-operads (equivalently, collapse of the spectral sequence), using the Fulton--MacPherson compactification. 
	
	Our Theorem \ref{thmA} then gives a dual spectral sequence. All intervals in the poset of trees are boolean lattices, and the spectral sequence of Theorem \ref{thmA} takes the simple form
	$$E_1^{pq}=\bigoplus_{\#\mathrm{Vert}(T)=p} H^{q}_c\left(\prod_{v \in \mathrm{Vert}(T)} X(n_v),\Q\right) \implies H^{p+q-1}_c(Y(n),\Q). $$
	This is now an \emph{operad} in the category of spectral sequences, whose $E_\infty$ page is the associated graded for a filtration on $H^{\bullet-1}_c(Y(n),\Q)$, and whose $E_1$ page is exactly the \emph{cobar construction} on the operad $H^\bullet_c(X(n),\Q)$. 
	
	Thus we see that working with compact support cohomology gives a quite general setting for proving bar/cobar-duality results for such pairs of operads $X$, $Y$.
\end{ex}

\begin{ex}\label{cm}
	Suppose that the poset $P$ is Cohen--Macaulay, or more generally that $P$ is graded with rank function $\rho$ and that $ \widetilde H_i(0,\beta) = 0$ for $i<\rho(\beta)-2 = \dim \N(0,\beta)$. Then if we apply the spectral sequence for the function $\rho$, the spectral sequence simplifies to 
	$$E_1^{pq} = \bigoplus_{\substack{\beta \in P\\ \rho(\beta)={p}}}  \widetilde H^{p-2}(0,\beta) \otimes  H^{q}_c(\overline S_\beta, i_\beta^\ast \F) \implies H^{p+q}_c(S_0,j_0^{-1}\F).$$
	In fact, something stronger is true: the chain complex $L^\bullet(\F)$ is filtered quasi-isomorphic to a complex of sheaves $K^\bullet(\F)$, with
	$$ K^d(\F) = \bigoplus_{\substack{\beta \in P\\ \rho(\beta)={d}}} \widetilde H^{d-2}(0,\beta)\otimes (i_\beta)_\ast (i_\beta)^\ast \F,$$
	and which is filtered by the ``stupid filtration''. The Cohen--Macaulay condition is extremely well studied and is known for large classes of posets. See e.g.\ \cite[Section 4]{wachsposettopology}.
	\end{ex}
	
	\begin{ex}
		Suppose that $X$ is a complex manifold, and that we are given an ``arrangement-like'' divisor $D$ on $X$, i.e.\ $D$ can locally be defined by a product of linear forms. Then the poset of strata is a geometric lattice and therefore Cohen--Macaulay. The complex of sheaves $K^\bullet(\F)$ is the Verdier dual of the one constructed by Looijenga in \cite[Section 2]{looijengaM3}. As part of his construction, he needs to inductively choose a certain free $\Z$-module $E_S$ for each stratum $S$ --- the fact that such a choice is possible is not obvious, and requires the Cohen--Macaulay condition!	\end{ex}
	
\begin{ex}\label{crit}
Suppose that $P$ is Cohen--Macaulay and the two contravariant functors $P \to \mathrm{Ch}_k$ given by $\alpha \mapsto R\Gamma_c(\overline S_\alpha,i_\alpha^\ast\F)$ and $\alpha \mapsto H^\bullet_c(\overline S_\alpha,i_\alpha^\ast \F)$ are quasi-isomorphic. Then the spectral sequence degenerates at $E_2$. 
	
	Indeed, we can realize $R\Gamma_c(X,K^\bullet(\F))$ as a double complex, with each column a direct sum of complexes $R\Gamma_c(\overline S_\alpha,i_\alpha^\ast\F)$, and the differentials in each row given by the differentials in the complex $K^\bullet(\F)$. If we have such a quasi-isomorphism we can therefore replace this double complex with one in which all vertical differentials vanish. Our spectral sequence is the spectral sequence given by filtering this double complex column-wise, since $K^\bullet(\F)$ has the stupid filtration. Thus it will indeed be the case that the spectral sequence has nontrivial differential only on $E_1$. 
	
	Suppose that each closed stratum $\overline S_\alpha$ is a compact complex manifold on which the $dd^c$-lemma holds, e.g. a K\"ahler or Moishezon manifold, and that the sheaf $\F$ is the constant sheaf $\R$. Then the criterion stated in the first sentence of this example is satisfied. Indeed, we may take as our model for $R\Gamma_c(\overline S_\alpha,\R) = R\Gamma(\overline S_\alpha,\R) $ the real de Rham complex of $\overline S_\alpha$, and then the validity of the above criterion is a particular case of \cite[Section 6, Main Theorem (ii)]{dgms}.
\end{ex}

\begin{ex}
	Suppose that $P$ is Cohen--Macaulay, and that each closed stratum is an algebraic variety whose compact support cohomology is of pure weight in each degree (e.g.\ a smooth projective variety). Then the spectral sequence also degenerates at $E_2$, using instead a weight argument. 
\end{ex}

\begin{ex}\label{confex}
	For a space $M$, let $F(M,n)$ denote the configuration space of $n$ distinct ordered points on $M$. If $M$ is an oriented manifold, a spectral sequence calculating the cohomology of $F(M,n)$ was constructed by Cohen and Taylor \cite{cohentaylor}. Their construction was later simplified by Totaro, who noticed that the spectral sequence is just the Leray spectral sequence for the inclusion $j:F(M,n) \hookrightarrow M^n$ \cite{totaro}. Getzler \cite{getzler99} then realized that the spectral sequence exists for a more or less arbitrary topological space, if one works with \emph{compactly supported} cohomology instead: more precisely, Getzler constructed a complex of sheaves quasi-isomorphic to $j_!j^{-1}\F$, whose compactly supported hypercohomology spectral sequence was Poincar\'e dual to the spectral sequence of Cohen--Taylor in the case of an oriented manifold.
	
	So suppose that $X=M^n$ for some space $M$, and let us stratify $X$ according to points coinciding. Then the poset of strata is the partition lattice $\Pi_n$, which is Cohen--Macaulay. Our complex $K^\bullet(\F)$ is exactly the one considered by Getzler, and the resulting spectral sequence
	$$E_1^{pq} = \bigoplus_{\substack{\beta \in \Pi_n\\ \rho(\beta)={p}}} \widetilde H^{p-2}(0,\beta) \otimes H^{q}_c(M^{n-p},\Z) \implies H^{p+q}_c(F(M,n),\Z)$$ 
	is the Poincar\'e dual of Cohen--Taylor's if $M$ is an oriented manifold. To see the identification of $K^\bullet(\F)$ with Getzler's resolution we need to know the cohomology of the partition lattice.
	
	Note first of all that each lower interval $[0,\beta]$ in the partition lattice is itself a product of partition lattices: e.g.\ if $\beta$ corresponds to the partition $(136\vert 27 \vert 45)$, then $[0,\beta] \cong \Pi_3 \times \Pi_2 \times \Pi_2$. Thus by the K\"unneth theorem we only need to know the top cohomology group $\widetilde H^{n-3}(\Pi_n,\Z)$. This calculation is hard to attribute correctly --- it follows by combining the results of \cite{klyachko} and \cite{stanleygroupactionsonposets}, see also \cite[Section 4]{joyalanalyticfunctors}. The result is in any case that $\widetilde H^{n-3}(\Pi_n,\Z)$ has rank $(n-1)!$ and that as a representation of the symmetric group $\sym_n$, it is isomorphic to $\mathsf{Lie}(n) \otimes \sgn_n$, where $\mathsf{Lie}(n)$ is the arity $n$ component of the Lie operad. But the same is also true for the cohomology group $H^{n-1}(F(\C,n),\Z)$, by the results of Cohen \cite{cohenladamay}; specifically, since the homology of the little disk operad is the Gerstenhaber operad, and the Gerstenhaber operad in top degree is just a suspension of the Lie operad, we get the above identification. This explains why the cohomology groups of $F(\C,n)$ appear in Getzler's construction of the resolution: the decomposition of $H^\bullet(F(\C,n))$ into summands corresponding to different partitions of $\{1,\ldots,n\}$ used by Getzler corresponds to
	\[ H^k(F(\C,n),\Z) \cong \bigoplus_{\substack{\beta \in \Pi_n\\ \rho(\beta)={k}}} \widetilde H^{k-2}(0,\beta).  \] 
	If we instead consider the second ``variant'' spectral sequence described in Example \ref{variants}, applied to the stratification of $X=M^n$ according to points coinciding, then we recover the spectral sequence of Bendersky--Gitler \cite{benderskygitler}.  
	\end{ex}

\begin{ex}
	Consider the example $X=M^n$ of a configuration space. Let $A_c^\bullet(M)$ be a cdga model for the compactly supported cochains on $M$. Then the criterion described in Example \ref{crit} is equivalent to $A_c^\bullet(M)$ being formal, i.e.\ that $A_c^\bullet(M)$ and $H^\bullet_c(M)$ are quasi-isomorphic. Hence if $A_c^\bullet(M)$ is formal then the Cohen--Taylor--Totaro spectral sequence degenerates after the first differential. 
	
	If in the same situation we consider the second variant of Example \ref{variants} (i.e. the Bendersky--Gitler spectral sequence), then we see by the same argument that the spectral sequence degenerates after the first differential whenever $M$ is a formal space, a result which is also proven in Bendersky--Gitler's original paper.\end{ex}

\subsection{Compatibility with Hodge theory}

We have already mentioned several times that in the algebraic case, we obtain a spectral sequence of $\ell$-adic Galois representations. It is natural to ask whether in the complex algebraic setting we get a spectral sequence of mixed Hodge structures. 

The answer is: yes, and it follows from Saito's theory of mixed Hodge modules \cite{saitomixedhodge}. Yet some care must be taken here. Saito proves the existence of a six functors formalism on the level of the derived categories $D^b(\MHM(X))$, where $X$ is a complex algebraic variety. We defined the complex $L^\bullet(\F)$ in such a way that $L^d(\F)$ is a sum of objects of the form $(i_{\alpha_d})_\ast (i_{\alpha_d})^\ast  \F$. Now if $\F$ is a mixed Hodge module then $(i_{\alpha_d})_\ast (i_{\alpha_d})^\ast  \F$ is in general only going to be an object of $D^b(\MHM(X))$, and this is not good enough: a ``chain complex'' of objects in a triangulated category $ T$ can in general have several non-isomorphic totalizations to an object of $ T$, or none at all. %(cf.\ e.g.\ \cite{kapranovderived}). 

Thus the construction can only be carried out if $i_\ast i^\ast$ is a $t$-exact functor, for $i$ a closed immersion. This is of course true for the usual $t$-structure of constructible sheaves, but it is \emph{false} for the perverse $t$-structure on $D^b_c(X)$: $i_\ast$ is still $t$-exact, but $i^\ast$ clearly is not. Since a mixed Hodge module does not have an underlying constructible sheaf but instead an underlying perverse sheaf, $i^\ast$ is not $t$-exact for mixed Hodge modules. 

However, one can choose instead to give $D^b(\MHM(X))$ a constructible (i.e.\ non-perverse) $t$-structure, which is uniquely characterized by the functor $\mathrm{rat} \colon D^b(\MHM(X)) \to D^b_c(X)$ being $t$-exact for the constructible $t$-structure on $D^b_c(X)$ \cite[Remark 4.6]{saitomixedhodge}. In other words, an object $\F \in D^b(\MHM(X))$ is in the heart of the constructible $t$-structure if and only $\mathrm{rat}(\F)$ is quasi-isomorphic to a constructible sheaf. In particular, $i^\ast$ will be $t$-exact for this $t$-structure, and $i_\ast$ will be $t$-exact whenever $i$ is a closed immersion.

 Let $\H(X)$ be the heart of the constructible $t$-structure of $D^b(\MHM(X))$, and let $\F$ be an object of $\H(X)$. Then for $d \geq 0$ we obtain an object 
$$ L^d(\F) = \bigoplus_{0 = \alpha_0 < \alpha_1 < \ldots < \alpha_d \in P}  (i_{\alpha_d})_\ast (i_{\alpha_d})^\ast  \F$$
of $\H(X)$, and we can define a differential $ L^d(\F) \rightarrow L^{d+1}(\F)$ just as before. Thus we get an object $L^\bullet(\F)$ of $D^b( \H(X))$, quasi-isomorphic to $j_{0!}^{}j^{-1}_0 \F \in \H(X)$. Moreover, we obtain a filtration of $L^\bullet(\F)$ by the same procedure as before. 
%This filtration allows us to write down a \emph{right Postnikov system} in the triangulated category $D^b(\H(X))$, with totalization $L^\bullet(\F)$. %making $L^\bullet(\F)$ into an object of the \emph{filtered derived category} $DF^b (\H(X))$. 
%
This filtration allows us to write down a \emph{Postnikov system} in the triangulated category $D^b(\H(X))$, with totalization $L^\bullet(\F)$ \cite[p. 262]{methodsofhomologicalalgebra}:
%
%\begin{tikzcd}[column sep=-0.5cm]
%	& & & \gr_F^p L^\bullet(\F)[-1]\arrow[rd] & & \gr_F^{p-1} L^\bullet(\F)[-1]\arrow[rd] & &\gr_F^{p-2} L^\bullet(\F)[-1] \arrow[rd] & \\
%	\ldots & \hspace{0.5cm}&  F^p L^\bullet(\F)  \ar[ru, "{[1]}"] & & \arrow[ll] \ar[ru, "{[1]}"] F^{p-1} L^\bullet(\F) & & \ar[ru, "{[1]}"] F^{p-2} L^\bullet(\F)\arrow[ll]	& & \arrow[ll] F^{p-3} L^\bullet(\F) & \hspace{0.5cm} & \ldots
%\end{tikzcd}
\begin{equation*}
\raisebox{-0.5\height}{\includegraphics{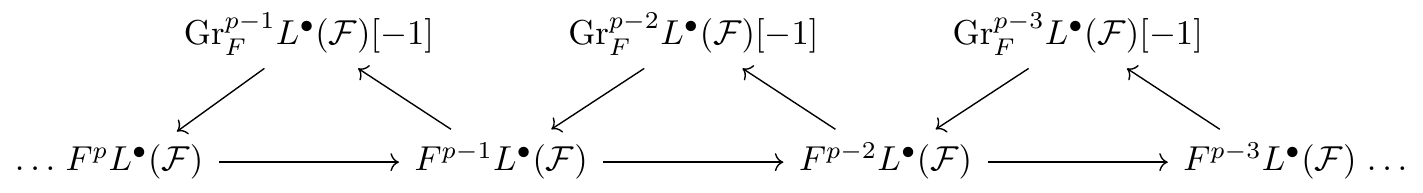}}
\end{equation*}

%\begin{tikzcd}[column sep=-0.5cm]{} &{} &{} & \gr_F^{p-1} L^\bullet(\F)[-1]\arrow[dl] &{} & \gr_F^{p-2} L^\bullet(\F)[-1]\arrow[dl] &{} &\gr_F^{p-3} L^\bullet(\F)[-1] \arrow[dl] &{} \\ \ldots & \hspace{0.5cm}&  F^p L^\bullet(\F) \arrow[rr]   &{} & \arrow[lu] \arrow[rr] F^{p-1} L^\bullet(\F) &{} & \arrow[lu] F^{p-2} L^\bullet(\F)\arrow[rr]	&{} &  \arrow[lu]F^{p-3} L^\bullet(\F) & \hspace{0.5cm} & \ldots
%\end{tikzcd}

%\begin{tikzcd}[column sep=-0.5cm]{} &{} &{} & \gr_F^{p-1} L^\bullet(\F)[-1]\arrow[dl] &{} & \gr_F^{p-2} L^\bullet(\F)[-1]\arrow[dl] &{} &\gr_F^{p-3} L^\bullet(\F)[-1] \arrow[dl] &{} \\ \ldots & \hspace{0.5cm}&  F^p L^\bullet(\F) \arrow[rr]   &{} & \arrow[lu, "{[1]}"] \arrow[rr] F^{p-1} L^\bullet(\F) &{} & \arrow[lu, "{[1]}"] F^{p-2} L^\bullet(\F)\arrow[rr]	&{} &  \arrow[lu, "{[1]}"]F^{p-3} L^\bullet(\F) & \hspace{0.5cm} & \ldots
%\end{tikzcd}

Let $\mathsf T$ be a triangulated category with a $t$-structure, and $\mathsf T^\heartsuit$ its heart. A \emph{realization functor} is an exact functor $D^b(\mathsf T^\heartsuit) \to \mathsf T$ whose restriction to the full subcategory $\mathsf T^\heartsuit$ is the inclusion into $\mathsf T$. If $\mathsf T$ is itself the derived category of an abelian category (not necessarily with its standard $t$-structure), then a realization functor always exists \cite[Section 3]{bbd}; more generally, a realization functor always exists if $\mathsf T$ is the homotopy category of a stable $\infty$-category. In particular we obtain a realization functor
$$ \mathsf{real} : D^b(\H(X)) \to D^b(\MHM(X)). $$
Exact functors preserve Postnikov system, and we get a Postnikov system in $D^b(\MHM(X))$ whose terms are of the form $\mathsf{real}(\gr_F^{p} L^\bullet(\F))$, up to a degree shift. Now we note that the functor $\mathsf{real}$ commutes with tensoring with a bounded complex of free abelian groups. As such the terms in the Postnikov system in $D^b(\MHM(X))$ are given by 
\begin{align*}
\mathsf{real}(\gr_F^{p} L^\bullet(\F)) &\cong \mathsf{real} \bigoplus_{\substack{\beta \in P\\ \sigma(\beta)={p}}} \widetilde C^{\bullet+2}(0,\beta) \otimes (i_\beta)_\ast (i_\beta)^\ast \F \\
 &\cong  \bigoplus_{\substack{\beta \in P\\ \sigma(\beta)={p}}} \widetilde C^{\bullet+2}(0,\beta) \otimes  \mathsf{real} (i_\beta)_\ast (i_\beta)^\ast \F \\
 &\cong  \bigoplus_{\substack{\beta \in P\\ \sigma(\beta)={p}}} \widetilde C^{\bullet+2}(0,\beta) \otimes (i_\beta)_\ast (i_\beta)^\ast \F,  
\end{align*} 
using in the last step that $(i_\beta)_\ast (i_\beta)^\ast \F$ is in $\H(X)$. Applying $Rf_!$ to this Postnikov system, where $f \colon X \to \mathrm{Spec}(\C)$ is the projection to a point, gives a Postnikov system in the derived category of mixed Hodge structures (the category of mixed Hodge modules over a point). The associated spectral sequence \cite[p. 263]{methodsofhomologicalalgebra} is the one of Theorem \ref{mainthm}(ii), now equipped with the canonical mixed Hodge structure coming from the fact that $\F$ is a mixed Hodge module.

\begin{rem}It seems likely that $D^b(\H(X)) \to D^b(\MHM(X))$ is an equivalence of categories, which would be the analogue for mixed Hodge modules of Be{\u\i}linson's theorem that the realization functor from the derived category of perverse sheaves to the derived category of constructible sheaves is an equivalence \cite{beilinsonperverse}, but I do not know if this is known and I have not attempted to prove it.
\end{rem}		
\section{Representation stability}
The notion of \emph{representation stability} was introduced by Church and Farb \cite{churchfarbrepresentationstability} as an extension of homological stability to situations where the Betti numbers do not actually stabilize. Roughly, a sequence $\{V(n)\}$ of representations of $\sym_n$ over $\Q$ is said to be \emph{representation stable} if, for $n \gg 0$, the representation $V(n+1)$ is obtained from the representation $V(n)$ by adding a single box to the top row of the Young diagram of each irreducible representation occuring in $V(n)$. Thus $V(n+1)$ is completely determined from $V(n)$ for sufficiently large $n$. Note in particular that the $\sym_n$-invariants satisfy $V(n+1)^{\sym_{n+1}} \cong V(n)^{\sym_n}$ for $n \gg 0$; if $\{V(n)\}$ were a sequence of homology groups, the $\sym_n$-invariants would satisfy homological stability in the usual sense.

The theory was clarified by the introduction of \emph{FI-modules} \cite{churchellenbergfarb}. The key point is that the underlying sequence of $\sym_n$-representations of an FI-module in $\Q$-vector spaces is representation stable if and only if the FI-module is \emph{finitely generated}. Most examples of representation stability arise from an FI-module in this way. 

One of the main examples of representation stability is given by the following theorem of Church \cite{church}: if $M$ is an oriented manifold, then $H^i(F(M,n),\Q)$ is a representation stable sequence of $\sym_n$-representations for any $i$. In this example, it was known since \cite{mcduff} that the cohomology of the unordered configuration space $F(M,n)/\sym_n$ satisfies homological stability for \emph{integer} coefficients if $M$ is an \emph{open} manifold, but also that integral homological stability is false in general. Church's result shows in particular that with $\Q$-coefficients, the unordered configuration space always satisfies homological stability. 

This result of Church fits well into the general framework of FI-modules. The assignment $S\mapsto F(M,S)$ is a contravariant functor from $\FI$ to spaces: if $S \subset T$, then $F(M,T) \to F(M,S)$ is the map that forgets all the points indexed by elements of $T \setminus S$. On applying $H^i(-,\Q)$ one gets an FI-module, which turns out to be finitely generated; in fact, finite generation holds already with integral coefficients \cite{cefn}. 

For the remainder of this paper, we will prove a theorem extending Church's result in several ways:
\begin{enumerate}
	\item Our proof works in a uniform way for a much larger class of configuration-like spaces, such as ``$k$-equals'' configuration spaces, the spaces $\overline w^c_\lambda(M)$ considered by Vakil--Wood, etc.
		\item We give a proof valid also in the algebro-geometric setting, so we get e.g.\ representation stability in the category of $\ell$-adic Galois representations. (This was previously proven for the spaces $F(M,n)$ in \cite{wolfsonfarb} under more restrictive assumptions on $M$.)
	\item We allow $M$ to have singularities. In the paper we focus on the case when $M$ is an algebraic variety (with arbitrary singularities); we comment towards the end on the differences in the topological setting. 
\end{enumerate}
In order to have representation stability in the more general setting of a singular space, one needs to work with Borel--Moore homology/compactly supported cohomology. Note that compactly supported cohomology is only contravariant for \emph{proper} maps, and the map $F(M,T) \to F(M,S)$ is (almost) never proper, so $H^\bullet_c(F(M,S),\Q)$ is not directly an FI-module. This should in any case not be surprising: if we want to recover Church's theorem by Poincar\'e duality when $M$ is an oriented manifold, then we had better prove that the cohomology $H^\bullet_c(F(M,S),\Q)$ satisfies representation stability up to a degree shift by the dimension of $F(M,S)$.

\subsection{Twisted commutative algebras and FI-modules}
Our proof of representation stability uses the formalism of FI-modules and twisted commutative algebras. We briefly recall the definitions for the reader's convenience. 
\newcommand{\CC}{\mathcal C}
Let $\CC$ be a symmetric monoidal category (the reader is encouraged to take $\CC$ to be the category of dg vector spaces over a field of characteristic zero). By a \emph{species in $\CC$} we mean a functor $\B \to \CC$, where $\B$ the category of finite sets and bijections. The category of species is equivalent to the category of sequences of representations of the symmetric groups $\sym_n$ in $\CC$. We write a species as ``$S \mapsto A(S)$'' or ``$n \mapsto A(n)$'', depending on whether we wish to consider is as a functor of finite sets or as a sequence of representations. We call $A(n)$ the \emph{arity $n$} component of the species $A$.

%We define a tensor product on the category of species by Day convolution: 
%$$(A \otimes B)(S) = \bigoplus_{S=T_1 \sqcup T_2} A(T_1)\otimes B(T_2).$$
Let us consider $\B$ as a symmetric monoidal category, with monoidal structure given by disjoint union. A \emph{twisted commutative algebra} (tca) in $\CC$ is a lax symmetric monoidal functor $\B \to \CC$. Thus a twisted commutative algebra in $\CC$ consists of a sequence $\{A(n)\}$ of $\sym_n$-representations in $\CC$, equipped with multiplication maps $A(n) \otimes A(m) \to A(n+m)$ which are $\sym_n \times \sym_m$-equivariant and satisfy suitable commutativity and associativity axioms. An equivalent definition is that a tca is a left module over the commutative operad $\Com$ in $\CC$. A third equivalent definition is that a tca is an algebra over the operad $\Com$ in the category of species in $\CC$, where the tensor product on the category of species is given by Day convolution:
$$(A \otimes B)(S) = \bigoplus_{S=T_1 \sqcup T_2} A(T_1)\otimes B(T_2).$$

Suppose that $\CC$ is the category of dg $R$-modules. Let $A$ be a species in $\CC$. We define the \emph{suspension} $\S A$ by
$$ \S A(n) = A(n)[-n] \otimes \mathrm{sgn}_n.$$The suspension is a symmetric monoidal endofunctor on the symmetric monoidal category of species. In partiular, if $A$ is a tca, then so is $\S A$. 

Let $\FI$ denote the category of finite sets and injections. By an \emph{FI-module} in the category $\CC$ we mean a functor $\FI \to \CC$. 

Let $\Com$ be the twisted commutative algebra which is given by the monoidal unit with trivial $\sym_n$-action in each arity, and for which all multiplication maps $\Com(S) \otimes \Com(T) \to \Com (S \sqcup T)$ are given by the canonical isomorphism $\mathbf 1 \otimes \mathbf 1 \cong \mathbf 1$. In other words, we are considering the commutative operad as a left module over itself. There is a general notion of a module over an algebra over any operad, which in this case specializes to an evident notion of a module over a tca. 

\begin{lem}
	Every module over the tca $\Com$ is in a canonical way an FI-module, and vice versa.
\end{lem}

\begin{proof}
	(Sketch) Let $M$ be a module over the tca $\Com$, and let $S \subset T$. Then we have a map
	$$ M(S) = M(S) \otimes \mathbf 1 = M(S) \otimes \Com(T \setminus S) \to M(S \sqcup (T \setminus S)) = M(T),$$
	where $\mathbf 1$ denotes the monoidal unit in $\CC$ and ``$=$'' denotes a canonical isomorphism. This makes $M$ into an FI-module. The converse construction is similar.
\end{proof}

If $A$ is a tca in $\CC$, then the choice of a morphism $a \colon \mathbf 1 \to A(1)$ is the same as the choice of a tca morphism  $\Com \to A$. Thus the choice of such an $a$ defines the structure of FI-module on the underlying species of the tca $A$. 

In particular, let $\CC$ be the category of graded $R$-modules. For a tca $A$ in $\CC$, let us write $A^i(n)$ for the degree $i$ component of $A(n)$. Then any $a \in A^0(1)$ defines a structure of FI-module on the collection $A^i(n)$, for all $i \in \Z$. We write $\vert a \vert$ for the degree of a homogeneous element in a graded vector space.

\begin{lem}
	Suppose $A$ is a tca in graded $R$-modules, and that $\{a_0,a_1,a_2,\ldots,\}$ is a set of generators. Suppose that $a_0\in A^0(1)$, $|a_i|<0$ for $i>0$, and
	$\lim_{i \to \infty} |a_i| = - \infty$.
	Then the FI-module $n \mapsto A^i(n)$ defined by multiplication with $a_0$ is finitely generated for all $i \in \Z$.
\end{lem}

\begin{proof}The hypotheses imply that there are only finitely many monomials in $\{a_1,a_2,\ldots,\}$ (i.e.\ all generators except $a_0$) of given degree. Those monomials of degree $i$ generate the FI-module $n \mapsto A^i(n)$.
\end{proof}

Applying the previous lemma to the $d$-fold suspension $\S^d A$ we get the following:

\begin{lem}\label{tcalemma}
	Suppose $A$ is a tca in graded $R$-modules, and that $\{a_0,a_1,a_2,\ldots,\}$ is a set of generators. Suppose that if $a_k \in A^i(n)$ then $i\leq nd$, with equality only for a single element $a_0 \in A^{d}(1)$, and that for each $p \in \Z$, there are only finitely many $k$ for which $a_k \in A^i(n)$ and $i \geq nd-p$. If $d$ is even, then $n \mapsto A^{i+dn}(n)$ becomes an FI-module by multiplication with $a_0$, and if $d$ is odd, then $n \mapsto A^{i+dn}(n) \otimes \mathrm{sgn}_n$ becomes an FI-module in this way. This FI-module is finitely generated for all $i \in \Z$.
\end{lem}

\subsection{The ``$\A$-avoiding'' configuration spaces}

For each finite set $S$, let $M^S$ be the cartesian product of $\vert S\vert$ copies of some space $M$. The functor $S \mapsto M^S$ is a co-FI-space. If $S \hookrightarrow T$, then we denote by $\pi_S^T : M^T \to M^S$ the projection. 

Let $\A$ be a finite collection of closed subspaces $\{A_i \subset M^{S_i}\}_{i=1}^\ell$, where each $S_i$ is some finite set. %Let $\mathbf \Delta$ be the smallest co-FI-subspace of $M^\bullet$ containing this collection of closed subspaces. We define 
%$$ F_\A(M,S)=M^S \setminus \mathbf \Delta(S)$$
%
%Consider some finite collection $\A$ of closed subspaces $A_1 \subset M^{n_1}$, $A_2 \subset M^{n_2}$, \ldots, $A_\ell \subset M^{n_\ell}$. For each $n \geq 0$ consider the stratification of $M^n$ consisting of all subspaces of the form $A_i \times M^{n-n_i}$, all images of such subspaces under the $\sym_n$-action, and all intersections of these spaces. Let $F_\A(M,n)$ be the unique open stratum.
%
%
For every finite set $T$, consider the stratification of $M^T$ given by all subspaces $$(\pi_{S_i}^T)^{-1}(A_i) \subset M^T, \qquad i=1,\ldots,\ell$$ ranging over  all inclusions $S_i \hookrightarrow T$, and all intersections of those subspaces. Let $P_\A(T)$ be the poset of strata in this stratification, and let $F_\A(M,T)$ denote the open stratum which is the complement of all of the $(\pi_{S_i}^T)^{-1}(A_i)$.

\begin{ex}
	If $\A$ is a singleton with $A_1 = \Delta \subset M^2$, then $F_\A(M,n)$ is the classical configuration space of $n$ points on $M$. If $\A$ instead consists only of the small diagonal in $M^k$, then $F_\A(M,n)$ is the ``$k$-equals'' configuration space of points on $M$. 
\end{ex}

\begin{ex}
	If a finite group $G$ acts on $M$, then we can let $\A$ consist of all subspaces $\{(x,g\cdot x) : x \in M\}$ inside $M^2$, in which case $F_\A(M,n)$ parametrizes $n$ distinct ordered points all of which are in distinct $G$-orbits. An example of this is the complement of hyperplanes in the Coxeter arrangement associated to the wreath product $(\boldsymbol{\mu}_r)^n \rtimes \sym_n$ acting on $\C^n$.
\end{ex}

\begin{ex}
	Suppose that $M = Y^2$ for some other space $Y$, and $\A$ consists of the collection $\{\Delta_{13},\Delta_{14},\Delta_{23},\Delta_{24}\}$ of diagonals inside $M^2 =Y^4$. Then $F_\A(M,n)$ parametrizes points $x_1,\ldots,x_n$ and $y_1,\ldots,y_n$ in $Y$ such that the $x_i$ (resp. $y_i$) may collide amongst each other, but $x_i \neq y_j$ for all $i,j$. %In particular we may set $Y=\C$ and divide by the action of $\sym_n$ this space was famously studied by Segal \cite{segalrationalfunctions}.
\end{ex}

\begin{ex}
	Let $\lambda$ be a partition of $n$. As our notation for partitions we use both $\lambda = (\lambda_1,\lambda_2,\lambda_3,\ldots)$ and $\lambda = (1^{n_1}\, 2^{n_2}\, 3^{n_3}\ldots )$, so that
	$n = \sum_{j \geq 1} \lambda_j = \sum_{i\geq 1} i \cdot n_i$. Vakil and Wood \cite{vakilwood} defined an open subspace $\overline w^c_\lambda(M) \subset M^n/\sym_n$, and they studied the behaviour of these spaces under the operation of ``padding $\lambda$ with ones'', i.e.\ letting $n_1$ approach $\infty$ while keeping all $n_i$, $i \geq 2$, fixed.  
	
	We can understand their construction in our terms as follows: if $\lambda$ is a partition of $n$ with $n_1=0$, let 
	$\Delta_\lambda \subset M^n$ be the locus where the first $\lambda_1$ points coincide, the subsequent $\lambda_2$ points coincide, etc., and put $\A = \{\Delta_\lambda\}$. If $\lambda'$ is the partition of $N \geq n$ obtained by padding $\lambda$ with ones, then $\overline w^c_{\lambda'}(M) = F_\A(M,N)/\sym_N$. In particular, rational homological stability for the spaces $\overline w_\lambda^c(M)$ under the operation of padding $\lambda$ with ones follows from representation stability for the spaces $F_\A(M,n)$ as $n \to \infty$. 
\end{ex}

\begin{ex}
	The configuration space of $n$ points in $\mathbb P^2$ such that no three of them lie on a line and no six lie on a conic is of the form $F_\A(M,n)$, where $M= \mathbb P^2$ and $\A$ has two elements which are closed subvarieties of $M^3$ and $M^6$, respectively. 
\end{ex}

We are going to prove a homological stability result for the spaces $F_\A(M,n)$. To avoid dealing with trivial cases we will assume that $\vert S_i \vert \geq 2$ for all $i$, and that no subspace $A_i$ can be written as 
\begin{equation*}
\label{cond}A_i = (\pi_{S_i'}^{S_i})^{-1}(A_i')
\end{equation*}
where $S_i'$ is a proper subset of $S$. 

\subsection{The set-up}\label{set-upsection}
Observe that there is an open embedding
$$ F_\A(M, S \sqcup T )\hookrightarrow F_\A(M,S) \times F_\A(M,T).$$
This makes $S\mapsto F_\A(M,S)$ a twisted cocommutative coalgebra of spaces. Since Borel--Moore homology is contravariant for open embeddings and admits cross products
$$ H_\bullet^{BM}(X,R) \otimes H_\bullet^{BM}(Y,R) \to H_\bullet^{BM}(X \times Y,R)$$
(which are isomorphisms if $R$ is a field), we get a twisted commutative algebra in graded $R$-modules:
$$ S \mapsto H_\bullet^{BM}(F_\A(M,S),R),$$
for any choice of coefficients $R$. %(Our only reason for using Borel--Moore homology rather than compact support cohomology is a general preference for working with algebras rather than coalgebras. Since in our cases of interest the coalgebra is finite dimensional over a field in each degree and conilpotent, the preference is purely psychological.)

The functor $S \mapsto P_\A(S)$ forms a twisted commutative algebra in the category of posets: the product of a stratum in $M^S$ and a stratum in $M^T$ is a stratum in $M^{S \sqcup T}$, which identifies $P_\A(S) \times P_\A(T)$ with an order ideal in $P_\A(S \sqcup T)$. If $\beta \in P_\A(S)$ and $\gamma \in P_\A(T)$, then we write $\beta \times \gamma$ for their product in $P_\A(S \sqcup T)$. 

Let $L^\bullet(S)$ denote the complex of sheaves $L^\bullet(R)$ on $M^S$ constructed in the previous part, associated to the stratification $P_\A(S)$. The previous paragraph implies that $L^\bullet(S) \boxtimes L^\bullet(T)$ is a subcomplex of $L^\bullet(S \sqcup T)$, and $\D L^\bullet(S) \boxtimes \D L^\bullet(T)$ is a quotient of $\D L^\bullet(S \sqcup T)$. Applying $R\Gamma(-)$, we see that the functor 
$$ S \mapsto R\Gamma^{-\bullet}(M^S,\D L^\bullet(S))$$
is a twisted commutative algebra of chain complexes, whose homology is the tca $ S \mapsto H_\bullet^{BM}(F_\A(M,S),R).$

\subsection{The hypotheses}

Let us now describe the hypotheses on $M$ and $\A$ that will lead to a proof of representation stability. Fix $M$ and $\A$ as above, and a coefficient ring $R$.  

\begin{hyp}\label{hyp}
	We assume that $H^{BM}_d(M,R)\cong R$, and that homology vanishes above this degree. We assume that (possibly after refining the stratifications) all strata in all spaces $M^n$ have finitely generated homology groups, and there exists an increasing function $\sigma \colon P_\A(n) \to \Z$ for all $n$ such that:
	\begin{enumerate}
		\item If $\beta \in P_\A(S)$ satisfies $\sigma(\beta)=p$ and $\gamma \in P_\A(T)$ satisfies $\sigma(\gamma)=q$, then $\sigma(\beta \times \gamma) = p+q$. 
		\item If $\beta \in P_\A(n)$ satisfies $\sigma(\beta)=p$, then $H_{i}^{BM}(\overline S_\beta,R)$ vanishes above degrees $dn - 2p$, and $H_{dn-2p}^{BM}(\overline S_\beta,R)$ is a projective $R$-module.
	\end{enumerate}	
\end{hyp}

\begin{ex}
	Suppose that $M$ is a geometrically irreducible algebraic variety of dimension $(d/2)$, and that $\A$ consists of closed subvarieties.  Then it will indeed be the case that $H^{BM}_d(M,\Z) \cong \Z$, $H^{BM}_i(M,\Z)=0$ for $i>d$, and all strata have finitely generated homology. Let $\sigma(\beta)$ be the codimension of $S_\beta$. After refining the stratifications we may assume all strata irreducible, in which case $\sigma$ becomes a strictly increasing function, and conditions (1) and (2) are clearly satisfied. 
\end{ex}

\begin{ex}Suppose that all the subspaces in $\A$ are given by diagonals, so all closed strata are products of the same space $M$. This covers e.g. all the configuration spaces considered by Vakil and Wood. In this case, for $\overline S_\beta \cong M^k \subset M^n$, we can take $\sigma(\beta)=(n-k)$. If we suppose that $M$ has finitely generated homology and finite dimension $d>1$, and $H_{d}^{BM}(M,R) \cong R$, then Hypothesis \ref{hyp} is satisfied. To verify the second condition, note that if  $\beta \in P_\A(n)$ satisfies $\sigma(\beta)=p$, then $\overline S_\beta \cong M^{n-p}$, whose highest nonzero Borel--Moore homology group is $H_{d(n-p)}^{BM}(M^{n-p},R) \cong R$. Since we assumed $d>1$, we get in particular vanishing above degree $dn-2p$ and that the homology group in degree $dn-2p$ is projective. 
\end{ex}

From now on we shall assume that Hypothesis \ref{hyp} is satisfied. 

\subsection{Proof with coefficients in a field}
In this subsection, we fix a field $k$ of coefficients, and all homology groups will be taken with coefficients in $k$. In the algebraic case we take $k=\Q_\ell$, where $\ell$ is not equal to the characteristic. We will later see that the proof can be modified to work also for integral coefficients, but the added complications arising from the lack of a K\"unneth isomorphism obscure the ideas somewhat.

\begin{lem}There exists a twisted commutative algebra of spectral sequences satisfying
	$$ E^1_{pq}(S) = \bigoplus_{\substack{\beta \in P_\A(S) \\ \sigma(\beta)=p}} \bigoplus_{i+j=p+q-2} \widetilde H_{i}(0,\beta) \otimes H_{j}^{BM}(\overline S_\beta,k), $$
	and which converges to the twisted commutative algebra $ S \mapsto H_\bullet^{BM}(F_\A(M,S),k).$
\end{lem}

\begin{proof}
	By Condition (2) in Hypothesis \ref{hyp}, the filtration on $L^\bullet(S) \boxtimes L^\bullet(T)$ induced by $\sigma$ agrees with the one on $L^\bullet(S \sqcup T)$, when we consider $L^\bullet(S) \boxtimes L^\bullet(T)$ as a subcomplex of $L^\bullet(S \sqcup T)$. This makes the twisted commutative algebra $S \mapsto R\Gamma^{-\bullet}(M^S,\D L^\bullet(S))$ a tca in filtered chain complexes, and the associated spectral sequence is given as above.
\end{proof}

%The functor $S \mapsto P_\A(S)$ forms a twisted commutative algebra in the category of posets: the product of a stratum in $M^S$ and a stratum in $M^T$ is a stratum in $M^{S \sqcup T}$, which identifies $P_\A(S) \times P_\A(T)$ with an order ideal in $P_\A(S \sqcup T)$. 
We say that an element $\beta \in P_\A(T)$ is \emph{indecomposable} if whenever $T= S \sqcup S'$ and $\beta$ is in the image of the multiplication map $P_\A(S) \times P_\A(S') \to P_\A(T)$, then $S$ or $S'$ is empty.  

\begin{lem}\label{codimlemma}
	There exists a constant $C$ such that if $\beta \in P_\A(T)$ is indecomposable, then $\sigma(\beta) \geq C \cdot \vert T \vert$. 
\end{lem} 

\begin{proof}
	The stratum $S_\beta$ is (an open stratum inside) an intersection of subspaces of the form $(\pi_{S_i}^T)^{-1}(A_i)$, for some collections of inclusions $S_i \hookrightarrow T$. We may assume this collection of subspaces to be irredundant. In order that $\beta$ be indecomposable, it is certainly necessary that the images of the $S_i \hookrightarrow T$ cover $T$, which means that the number of subspaces that one needs to intersect to obtain $S_\beta$ grows linearly in $\vert T\vert$. Moreover, since we assumed the collection irredundant and that $\sigma$ increasing, the value of $\sigma$ must go up by at least $1$ for each subspace we intersect, which proves the result. \end{proof}

Let $E$ be the twisted commutative algebra in graded vector spaces given by
$$ S \mapsto E_\bullet(S); \qquad E_i(S) = \bigoplus_{p+q=i} E^1_{pq}(S). $$

\begin{lem} The tca $E$ satisfies the hypotheses of Lemma \ref{tcalemma}, so that $\S^d E$ is a finitely generated FI-module.\end{lem}

\begin{proof}The tca $E$ is generated by classes $\widetilde H_{i}(0,\beta) \otimes H_{j}^{BM}(\overline S_\beta,k)$ where $\beta$ ranges over indecomposable elements of $P_\A(n)$. 
	
If $\beta \in P_\A(n)$ satisfied $\sigma(\beta)=p$, then $\widetilde H_{i}(0,\beta)$ vanishes above degree $p-2$, and     $H_{j}^{BM}(\overline S_\beta,k)$ vanishes above degree $dn-2p$. Thus the corresponding generators in $E_i(n)$ satisfy $i\leq dn-p$. In particular, we get a generator in degree $i=dn$ only for $p=0$. But the open stratum is indecomposable only when $n=1$, in which case we get a single generator in this degree from the one-dimensional space $H_{d}^{BM}(M,k)$.

Moreover, by Lemma \ref{codimlemma}, for each $p\geq 0$ there are only finitely many strata in $P_\A(n)$ (summed over all $n$) with $\sigma(\beta) \leq p$, which means that only finitely many of these generators satisfy $i\geq dn-p$. Thus the hypotheses of Lemma \ref{tcalemma} are satisfied.
	\end{proof}

%\begin{rem}
%	``$\mathrm{sgn}_n^{\otimes d}$'' is of course just a compact way of writing ``$\mathrm{sgn}_n$ if $d$ is odd, and the trivial representation if $d$ is even''. 
%\end{rem}

\begin{thm}\label{A}
	The FI-module given by $n \mapsto H^{BM}_{i+dn}(F_\A(M,n),k)\otimes \mathrm{sgn}_n^{\otimes d}$ is finitely generated for all $i \in \Z$. 
\end{thm}

\begin{proof}
	By the previous lemma, $$ n \mapsto (\S^d E)(n)  = \bigoplus_{p+q=i+dn} E^1_{pq}(n) \otimes \mathrm{sgn}_n^{\otimes d}$$ is a finitely generated FI-module. Then so is $ n \mapsto  \bigoplus_{p+q=i+dn} E^\infty_{pq}(n) \otimes \mathrm{sgn}_n^{\otimes d}, $ being a subquotient of a finitely generated FI-module \cite[Theorem 1.3]{churchellenbergfarb}. But the latter is just the associated graded of the FI-module $n \mapsto H^{BM}_{i+dn}(F_\A(M,n),k) \otimes \mathrm{sgn}_n^{\otimes d}$ for the Leray filtration.
\end{proof}
\begin{rem}The sign representation which appears for odd $d$ arises from Lemma \ref{tcalemma}, and does not play a role in the algebraic case since an algebraic variety has even (real) dimension. That the sign representation should appear is clear, if we want to recover representation stability for the usual cohomology $H^i(F(M,n),k)$ from Poincar\'e duality when $M$ is an oriented manifold. Indeed, the Poincar\'e duality isomorphism for $F(M,n)$ involves capping with the fundamental class, which generates the $1$-dimensional vector space $H_{dn}^{BM}(F(M,n),k) \cong H_{dn}^{BM}(M^n,k)$. When $d$ is even this vector space carries the trivial representation of $\sym_n$, but when $d$ is odd it has the sign representation.
	\end{rem}
\begin{rem}\label{homologymanifoldremark}
	A space $M$ is called an $R$-homology manifold of dimension $d$ if for all $x \in M$ one has $$H_i(M,M \setminus\{x\};R) \cong \begin{cases}
	R & i=d \\ 0 & i \neq d.
	\end{cases}$$
	Equivalently, $\D R \cong R[d]$. A trivial example is an oriented $d$-manifold for any $R$; a more typical example is that a complex algebraic variety with finite quotient singularities is a $\Q$-homology manifold (of dimension twice its dimension over $\C$). An $R$-homology manifold of dimension $d$ satisfies Poincar\'e duality in the form $H^{BM}_{i}(M,R) \cong H^{d-i}(M,R)$. If we add to Hypothesis \ref{hyp} the condition that $M$ is an $R$-homology manifold of dimension $d$, then the conclusion of Theorem \ref{A} (if $R$ is a field, or the results of the next subsection if $R$ is a PID) become equivalent to the claim that $n \mapsto H^i(F_\A(M,n),R)$ is a finitely generated FI-module. For instance, if $M$ is a connected oriented manifold of dimension $d > 1$, then $H^i(F(M,n),\Z)$ is a finitely generated FI-module for all $i$; this is how the results of \cite{church,cefn} can be obtained as specializations of those in this paper.
\end{rem}

\subsection{Proof for integral coefficients}

In the proof of Theorem \ref{A} in the preceding subsection, we used that any subquotient of a finitely generated FI-module is finitely generated. This was proven for field coefficients in \cite{churchellenbergfarb}, but the result was then extended to any noetherian ring in \cite{cefn}. We will now use the latter result to give a proof also for integral coefficients.

However, the real reason we used field coefficients in the preceding subsection was to have a robust K\"unneth isomorphism: without it, it would not be true that generators for the twisted commutative algebra we considered arise from indecomposable strata. Namely, if $\beta$ is decomposable --- say $\overline S_\beta = \overline S_\gamma \times \overline S_{\gamma'}$ --- then the cross product map $$H_\bullet^{BM}(\overline S_\gamma,\Z) \otimes H_\bullet^{BM}(\overline S_{\gamma'},\Z) \to H_\bullet^{BM}(\overline S_\beta,\Z)$$ is not necessarily surjective. To remedy this, we will need to work on the chain level, analogous to \cite[Lemma 4.1]{cefn}. 

For the remainder of this section we fix a coefficient ring $R$ which we assume to be a PID, e.g. $R=\Z$ or $R=\Z_\ell$. In order for the proof to work, we shall need to verify a refinement of Hypothesis \ref{hyp}.

\begin{lem}\label{hyp2}Assume Hypothesis \ref{hyp}. For each stratum $\overline S_\beta$ in any of the spaces $M^n$, we can choose a quasi-isomorphism 
	$$ C_\bullet(\overline S_\beta) \simeq R\Gamma^{-\bullet}(\overline S_\beta,\D R)$$
	where $C_\bullet(\overline S_\beta)$ is a bounded complex of finitely generated free modules, and such that
	\begin{itemize}
		\item For any decomposable stratum $\overline S_\alpha \times\overline S_\beta$, we have an equality
		$$C_\bullet(\overline S_\alpha) \otimes C_\bullet(\overline S_{\beta}) = C_\bullet(\overline S_\alpha \times\overline S_\beta)$$
		compatible with the quasi-isomorphism
		$$ R\Gamma^{-\bullet}(\overline S_\alpha,\D R) \otimes R\Gamma^{-\bullet}(\overline S_\beta,\D R) \simeq R\Gamma^{-\bullet}(\overline S_\beta \times \overline S_\alpha,\D R);$$
		\item For $\overline S_\alpha \subset \overline S_\beta$, there is a map
		$$ C_\bullet(\overline S_{\alpha }) \to C_\bullet(\overline S_\beta)$$
		compatible with the map 
		$$ R\Gamma^{-\bullet}(\overline S_\alpha,\D R) \to R\Gamma^{-\bullet}(\overline S_\beta,\D R); $$
		\item If $\overline S_\beta \subset M^n$ has $\sigma(\beta)=p$ then $C_\bullet^{BM}(\overline S_\beta,R)$ vanishes above degree $dn-2p$, and $C_d(M) \cong R$. 
	\end{itemize}
\end{lem}

It is immediate from Hypothesis \ref{hyp} that such a complex $C_\bullet(\overline S_\beta)$ can be constructed for each individual stratum, but we need the choices to satisfy various compatibilities.

Taking the lemma for granted for the moment, the idea will be to run nearly the same proof, but instead of starting at the $E^1$ page of the spectral sequence, we start at $E^0$. Equivalently, we work directly on the level of the double complex computing $R\Gamma^{-\bullet}(M^n,\D L^\bullet(R))$, associated to our filtration of $\D L^\bullet(R)$. The filtration on $\D L^\bullet(R)$ has its associated graded pieces quasi-isomorphic to sums of complexes of the form $\widetilde C_{-\bullet-2}(0,\beta) \otimes (i_\beta)_\ast \D R$, where $i_\beta$ is the inclusion of a closed stratum, so the columns of this double complex are of the form $R\Gamma^{-\bullet}(M^n,\widetilde C_{-\bullet-2}(0,\beta) \otimes (i_\beta)_\ast \D R)$. Under our Hypothesis \ref{hyp2} we may (for all $\beta$) replace this with the totalization of the double complex $\widetilde C_{\bullet+2}(0,\beta) \otimes C_{\bullet}(\overline S_\beta)$. Then this collection of double complexes becomes a twisted commutative algebra, and that we get a tca of spectral sequences:

\begin{lem}There exists a twisted commutative algebra of spectral sequences satisfying
	$$ E^0_{pq}(S) = \bigoplus_{\substack{\beta \in P_\A(S) \\ \sigma(\beta)=p}} \bigoplus_{i+j=p+q-2} \widetilde C_{i}(0,\beta) \otimes C_{j}(\overline S_\beta), $$
	and which converges to the twisted commutative algebra $ S \mapsto H_\bullet^{BM}(F_\A(M,S),R).$
\end{lem}

%\begin{lem}There exists a twisted commutative algebra of  double complexes $S \mapsto E_{\bullet,\bullet}(S)$, satisfying%
%	$$ E_{pq}(S) = \bigoplus_{\substack{\beta \in P_\A(S) \\ \sigma(\beta)=p}} \bigoplus_{i+j=p+q-2}   \widetilde C_{i}(0,\beta) \otimes C_{j}^{BM}(\overline S_\beta,R) , $$
%	and whose homology is the twisted commutative algebra $ S \mapsto H_\bullet^{BM}(F_\A(M,S),R).$
%\end{lem}

If we consider the twisted commutative algebra in graded abelian groups given by
$$ S \mapsto E_\bullet(S); \qquad E_i(S) = \bigoplus_{p+q=i} E_{pq}^0(S), $$
then this tca \emph{will} be generated by classes $\widetilde C_{i}(0,\beta) \otimes C_{j}^{BM}(\overline S_\beta,R)$ where $\beta$ ranges over indecomposable elements of $P_\A(n)$. In particular, it satisfies the hypotheses of Lemma \ref{tcalemma}, for the same reason as the tca $E$ considered in the previous subsection, and the rest of the proof carries over without any changes. 

Let us now argue that Lemma \ref{hyp2} is satisfied. 

\begin{proof} (of Lemma \ref{hyp2}) First off, we replace each $R\Gamma^{-\bullet}(\overline S_\beta,\D R)$ by a functorial free resolution, and then apply the ``wise'' truncation functor $\tau^{\leq (dn-2p)}$. Let us call the resulting complexes $C_\bullet^{BM}(\overline S_\beta,R)$. Unfortunately, the truncation functor has the wrong functoriality: there is for any chain complex $C_\bullet$ a map $C_\bullet \to \tau^{\leq n}C_\bullet$, but we need a map in the opposite direction. However, the assumption that $H_{dn-2p}^{BM}(\overline S_\beta,R)$ is projective and that $R$ is a PID implies that $C_{dn-2p}^{BM}(\overline S_\beta,R)$ is itself free. In particular, $C_\bullet^{BM}(\overline S_\beta,R)$ is itself a free resolution, and the truncation map has a section. One checks that any choice of section gives rise to a well defined chain map $C_\bullet^{BM}(\overline S_\beta,R) \to C_\bullet^{BM}(\overline S_\alpha,R)$ for $\overline S_\beta \subset \overline S_\alpha$, making the assignment $\beta \mapsto C_\bullet^{BM}(\overline S_\beta,R)$ functorial. 	
	
To replace these complexes with ones that are finitely generated in each degree, we work inductively, starting with $M$ itself. If we have chosen complexes $C_\bullet(\overline S_\beta) \stackrel \sim \to C_\bullet^{BM}(\overline S_\beta,R)$ for all $\beta \in P_\A(n)$, $n<N$, then the condition that we have an ``on-the-nose'' K\"unneth isomorphism determines our choice of $C_\bullet(\overline S_\beta)$ for all decomposable strata $\beta \in P_\A(N)$. Now I claim that if we have compatible choices of  $C_\bullet(\overline S_\beta)$ for all $\beta$ in some order ideal $I \subset P_\A(N)$, and $\alpha$ is any minimal element of $P_\A(N) \setminus I$, then we can also choose $C_\bullet(\overline S_\alpha)$ compatibly (and thus the inductive procedure can be continued). Indeed, consider the composition 
$$ \colim_{\beta < \alpha} (C_\bullet(\overline S_\beta)) \longrightarrow \colim_{\beta < \alpha}(C_\bullet^{BM}(\overline S_\beta,R)) \longrightarrow C_\bullet^{BM}(\overline S_\alpha,R). $$
We note that the colimit over a finite poset can be defined as the totalization of a functorially defined chain
complex, which is of finite rank in each degree if each chain complex in the colimit is. By induction, the image of this composition is then finitely generated in each degree, so we may choose a quasi-isomorphism $C_\bullet(\overline S_\alpha) \to C_\bullet^{BM}(\overline S_\alpha,R)$ where $C_\bullet(\overline S_\alpha)$ is again finitely generated in each degree, such that the image contains the image of $\colim_{\beta < \alpha} (C_\bullet(\overline S_\beta))$. Then there is a factorization $$\colim_{\beta < \alpha} (C_\bullet(\overline S_\beta)) \to C_\bullet(\overline S_\alpha) \stackrel \sim \to C_\bullet^{BM}(\overline S_\alpha,R)$$ which proves the claim.
\end{proof}

\bibliographystyle{alpha}
\bibliography{../database}

\end{document}